\documentclass[12pt]{article}      %11pt!!!

\usepackage{pslatex}        %Times font!!!

\usepackage{fancyhdr}
\usepackage{epsfig}
\usepackage{color}
\usepackage{natbib}

\usepackage{graphicx}
\usepackage{latexsym}
\usepackage{amssymb}
\usepackage{amsfonts}
\usepackage{amsmath}
\usepackage{amsthm}
\usepackage{mathtools}
\usepackage{mathabx}

\theoremstyle{plain}
\newtheorem{thm}{Theorem}

\DeclareMathOperator{\PP}{{\bf P}}
\DeclareMathOperator{\E}{{\bf E}}
\DeclareMathOperator{\Var}{{\bf Var}}
\DeclareMathOperator{\hE}{\widehat{\bf E}}
\DeclareMathOperator{\hS}{\widehat{\bf S}}

\def\boldhline{\noalign{\global\arrayrulewidth.8pt}\hline\noalign{\global\arrayrulewidth.4pt}}

\newcommand{\iid}{\mbox{i.i.d.\frenchspacing}}

\usepackage[top=1in, bottom=1in, left=0.78in, right=0.78in]{geometry}
\numberwithin{equation}{section}

\begin{document}

\LARGE
\begin{center}
\textbf {Gini's mean difference and variance as measures of finite populations scales }\\[2.25\baselineskip]

\small
\text{Andrius {\v C}iginas${}^{1}$ and Dalius Pumputis${}^{2}$}\\

\medskip

{\footnotesize
${}^{1}$Vilnius University Institute of Mathematics and Informatics, LT-08663 Vilnius, Lithuania \\
${}^{2}$Lithuanian University of Educational Sciences, LT-08106 Vilnius, Lithuania 
}\\[2.25\baselineskip]
\end{center}

\small
\begin{abstract}

We consider Gini's mean difference statistic as an alternative to the empirical variance in the settings of finite populations where simple random samples are drawn without replacement. In particular, we discuss specific (in the finite population context) estimation strategies for a scale of the population, related to the alternative statistic under possible presence of outliers in the data.

The paper presents also a wide comparative survey of properties of the Gini mean difference statistic and the empirical variance. It includes asymptotic properties of both statistics: the asymptotic normality, one-term Edgeworth expansions and bootstrap approximations for Studentized versions of the statistics. An estimation of the variances and other parameters of the statistics is also in the study, where we exploit an auxiliary information on the population elements in the case of its availability. Theoretical results are illustrated with a simulation study. 

\end{abstract}
\vskip 2mm

\normalsize


\vskip 2mm

\noindent\textbf{MSC classes:} 62E20

\section{Introduction}\label{s:1}

\let\thefootnote\relax\footnotetext{The research of the first author is supported by European Union Structural Funds project "Postdoctoral Fellowship Implementation in Lithuania".} 

Together with a location parameter, a spread  (or scale) of a survey population are usually the parameters of interest. If a statistician assumes the classical model of independent and identically distributed (\iid) observations, then, at least, he has at his disposal the number of parametric distributions families, e.g., Gaussian, Cauchy, etc. Assume that, he chooses the particular family for a further analysis of the data. This family comes with its own measures of location and scale, for instance, the normal distribution parameters `suggest' to measure the mean and variance of the survey population, and the Cauchy distribution is specified by the population median and interquartile range. The traditional statistics theory has the answers how to get efficient estimates of locations and scales under commonly used populations models. However, parametric statistics models, being comparatively convenient, are known also as non-robust, i.e., deviations from their assumptions may lead to misleading conclusions. As it is often an instance, an appearance of some or more outlying observations can strongly affect the quality of typical estimators of the population location and scale. Then, if we believe that these outliers are, e.g., measurement errors, robust estimation methods can be a treatment of the problem. The pioneering (in formalization of the robust estimation) book of \cite{H_1981} starts with the normal distribution scale estimation example showing an inefficiency of the empirical variance compared to the mean absolute deviation under the presence of outliers in the sample data. In a sense similar to the latter statistic is Gini's mean difference (GMD) statistic. This estimator, its properties, connections and comparisons with the sample variance is the aim of the present paper. 

We consider the GMD statistic as an alternative to the empirical variance in the setting of a finite population $\{ 1,\ldots, N\}$ of elements with the corresponding set of real values ${\cal X}=\{ x_1,\ldots, x_N\}$ of the variable $x$ under investigation, and for the simple random sample 
$\{ 1,\ldots, n\}$ of size $n<N$ \emph{drawn without replacement} from the population with the measurements $\mathbb X=\{ X_1,\ldots, X_n\}$ of the variable $x$. In particular, the population parameters
\begin{equation}\label{scale_param_G}
G={N\choose 2}^{-1} \sum_{1\leq i<j\leq N} |x_i-x_j| 
\end{equation}
and
\begin{equation}\label{scale_param_V}
V={N\choose 2}^{-1}\sum_{1\leq i<j\leq N} (x_i-x_j)^2/2
\end{equation}
are two candidates to measure a scale of ${\cal X}$, only the latter seems more natural because of $\Var X_1=(N-1)V/N$. The corresponding unbiased estimators of these parameters are the GMD statistic
\begin{equation}\label{GMD}
U_G={n\choose 2}^{-1} \sum_{1\leq i<j\leq n} |X_i-X_j|
\end{equation}
and the empirical variance
\begin{equation}\label{EV}
U_V={n\choose 2}^{-1}\sum_{1\leq i<j\leq n} (X_i-X_j)^2/2.
\end{equation}
As an alternative to \eqref{EV}, the GMD statistic, known better since \cite{G_1912},  is widely used in economics. Now it is an ordinary measure of a dispersion of a distribution of income and also in cases of similar variables, see monograph of \cite{YS_2013}, where, by words of the authors, the commonly used variance-based analyses are `translated' into Gini-based. A use of the GMD is not restricted with measurements of an economic inequality. As in problems of economists, where data deviate from the normality, the parameter $G$ and its estimator $U_G$ can be used as dispersion's measures for many kinds of statistical data. Our choice of the finite populations setting has a motivation from the side of economics too, because, in economical surveys, the number $N$ of surveyed objects or subjects is not necessarily so large (compared to the sample size) that to ignore a dependence between the observations in the set $\mathbb X$. 

In Section \ref{s:strateg}, we consider three estimation of the finite population scale strategies related to the alternative $U_G$. We exploit two assumptions, which are usually possible in the finite population context: the so-called superpopulation assumption, and the availability of an auxiliary information about the population elements. We perform also simulation experiments, where we analyze advantages and disadvantages of the strategies, and compare them under populations without and with outliers.

The GMD statistic is one of several well-known universal estimators as, e.g., the median absolute deviation, interquartile range, which are less sensitive to outliers than the sample variance. Looking from the side of the robust estimation theory, if we can link data to a parametric population model, then, in the particular situations, there are more effective robust estimators of scale than those common ones, see \cite{H_1981}. But we focus here on an unified improvement of the empirical variance.

The next premium, which should be paid, is a relatively complex access to properties of the GMD statistic. On the other hand, in these problems, $U_G$ is more attractive than the other mentioned examples of universal estimators because of its smoothness (in a certain sense) or that it uses the complete sample information. In Section \ref{s:2}, an asymptotic analysis of distributions of the statistics $U_G$ and $U_V$ shows that their properties are similarly simple. To explain it, we apply an available theory of $U$- and $L$-statistics in the case of samples without replacement. In particular, statistics \eqref{GMD} and \eqref{EV} are likely the most popular $U$-statistics of degree two, and \eqref{GMD} is also the $L$-statistic, see \cite{S_1980}. As the $L$-statistic, $U_G$ is smooth in the sense that its weight function is smooth, see ibidem.

To be consistent with already known results, first, we mark that expressions of the variance of $U_G$ and its approximations are known since \cite{N_1936} and \cite{L_1952} in the case of \iid{} observations, and since \cite{G_1962} for the simple random samples without replacement. Second, a strong method to study the variances and the asymptotic normality of the statistics $U_G$ and $ U_V$ is Hoeffding's decomposition for $U$-statistics in \cite{H_1948}. Much latter, in \cite{ZC_1990}, the analogous decomposition was used in the case of finite population. Third, similarly, second-order approximations theory for samples without replacement has been realized after the case of \iid{} observations: \cite{KW_1990}, and \cite{BG_1999} follow \cite{BGZ_1986} on one-term Edgeworth approximations to the distributions of standardized $U$-statistics; papers of \cite{B_2003}, and \cite{B_2007}, on an one-term Edgeworth expansion for Studentized $U$-statistics and bootstrap approximations, appeared after \cite{H_1991}.

Since the true values of variances of the statistics are almost always unknown, we prefer to consider the asymptotic normality, one-term Edgeworth expansions and bootstrap approximations for Studentized versions of the statistics $U_G$ and $U_V$. A basis for such a study is the general theory in \cite{BG_2001}, \cite{B_2003}, and \cite{B_2007} (with without-replacement bootstrap of \cite{BBH_1994}), where, to ensure a validity of the approximations, quite general smoothness conditions are imposed on parts of the Hoeffding decomposition of $U$-statistics. Theorems of Section \ref{s:2} let to compare distributional properties of $U_G$ and $U_V$ much easy. 

A successful application of the one-term Edgeworth expansion requires to have good estimators (in the sense of an asymptotic consistency or a small mean square error) of the expansion's parameters. In the case of symmetric statistics (symmetric functions of observations) including $U$-statistics, jackknife techniques are used to estimate these parameters, see \cite{PZ_1998}, and \cite{B_2001}. In the separate cases of statistics, for example, for $U_G$ and $U_V$, there are more ways to construct estimators of the Edgeworth expansions parameters, e.g., for $L$-statistics including $U_G$, the bootstrap was used in \cite{C_2013a} and, assuming that the auxiliary information is available, calibration methods were applied in \cite{PC_2013}. In Section \ref{s:4}, we propose simple and also efficient estimators of the parameters, without the auxiliary information and also using it. Similar estimators of the variances of $U_G$ and $U_V$ are also considered. In Section \ref{s:5}, we discuss empirical Edgeworth expansions, based on the estimators of the parameters, and bootstrap approximations. In Section \ref{s:6}, we compare the obtained estimation results for both statistics of interest in the simulation study. Here we are interested also in a role of outliers in populations. Conclusions of the paper are given in Section \ref{s:7}.

\section{Estimation of scale}\label{s:strateg}

\subsection{Outliers and estimation strategies}

In the \iid{} setup, for many common parametric models of populations, the sample variance is an efficient estimator under ideal or close to ideal conditions. But assume that some of the sample data differ substantially from the other. Then the GMD statistic can be a better choice because it puts smaller weights on extreme observations thus lowering their impact on the estimation.

 In the finite population case, outliers are less influential too, when $U_G$ is applied. To see it, let us write parameters \eqref{scale_param_G} and \eqref{scale_param_V} in the different form. Assume (here and further in the paper), without loss of generality, that $x_1\leq\cdots\leq x_N$, and denote $\Delta_i=x_{i+1}-x_i$, $i=1,\ldots, N-1$. Then, taking $x_j-x_i=\sum_{k=i}^{j-1}\Delta_k$, one can obtain
\begin{equation*}
G^2=\frac{4}{N^2(N-1)^2}\Bigg[ \sum_{i=1}^{N-1}i^2(N-i)^2\Delta_i^2+2\sum_{1\leq i<j\leq N-1}ij(N-i)(N-j)\Delta_i\Delta_j\Bigg]
\end{equation*}
and
\begin{equation*}
V=\frac{1}{N(N-1)}\Bigg[ \sum_{i=1}^{N-1}i(N-i)\Delta_i^2+2\sum_{1\leq i<j\leq N-1}i(N-j)\Delta_i\Delta_j\Bigg].
\end{equation*}
These expressions are connected via the formal transformation
\begin{equation}\label{transf_X}
\Delta'_i\Delta'_j=\frac{4j(N-i)}{N(N-1)}\Delta_i\Delta_j, \qquad 1\leq i\leq j\leq N-1
\end{equation}
of ${\cal X}$, where $\Delta'_i=x'_{i+1}-x'_i$, $i=1,\ldots, N-1$, which explains the assertion. We note that system of equations \eqref{transf_X} has not a solution ${\cal X}'=\{x'_1, \ldots, x'_N\}$ except in cases of very simple ${\cal X}$.
%pastaba: issprendziama, kai, pvz., X={0,0,1,1}.
\smallskip

\noindent\textbf{Outliers model.} For the simple random samples without replacement, we assume an existence of so-called representative outliers. This notion was introduced in \cite{C_1986}. It means the assumptions that: outlying observations are not errors of a measurement; the unsampled population part should contain outliers too. If these assumptions do not hold, then, in sample surveys, the problem of outliers is treated usually as a different from the estimation.  

More formally, denote by $0\leq p\leq N$ the number of outliers in the population. Assume that the population elements $\{ i_1, \ldots, i_p\}\subseteq \{ 1, \ldots, N\}$ belong to a different population, but this phenomenon is not known while the sample $\mathbb X$ was not obtained. Then the corresponding values from $x_1,\ldots, x_N$ are treated as outliers. In the random sample $\mathbb X$, the number of outliers is random and equals to the number of elements in the set  $\{ i_1, \ldots, i_p\}\cap \{ 1, \ldots, n\}$.

The proportion $p/N$ of outliers can be restricted without a significant loss of generality. In particular, as it is pointed in \cite{H_1981}, a part of gross errors (outliers) in samples usually is not larger than $10\%$. An interesting note on this issue is given in \cite{CF_1980}: "\emph{\ldots it is reasonable to consider three potential outliers in a data set of $10$ observations, but it is unrealistic to expect $30$ outliers out of a data set of $100$ observations. In the latter case, the outlier detection problem becomes one of discrimination between two or more classes of data.}". Similarly, for finite populations, if a large portion of outliers is expected in the population, they are neutralized typically (with a help of an auxiliary information) by applying stratified sampling designs, i.e., collecting potential outliers into a separate stratum. Another but similar solution, in this case, is a postratification.
\smallskip

\noindent\textbf{Estimation strategies.} Specific for the finite population ways to apply the GMD statistic as the alternative to the sample variance are the following.

\noindent ($S_1$) Assume that the fixed numbers $x_1,\ldots, x_N$ are the realizations of \iid{} random variables $X_1^*, \ldots, X_N^*$ (superpopulation model) from a parametric family of distributions with the scale parameter which is an one-argument function of $\sqrt{\Var X_1^*}$. Then the scale of ${\cal X}$ is treated as the same function of $\sqrt{V}$, and the estimator of the argument $\sqrt{V}$ is taken to be of the form $aU_G$, where $a>0$ is a constant compensating a bias.

\noindent ($S_2$) Under the presence of well-correlated and completely known auxiliary variable $z$ with the values ${\cal Z}=\{ z_1,\ldots, z_N\}$ in the population, the scale measure is $\sqrt{V}$ and its estimator is $aU_G$ with the correction $a>0$ evaluated from ${\cal Z}$.  

\noindent ($S_3$) The parameter $G$ is itself treated as the scale of ${\cal X}$, and the GMD statistic $U_G$ is its estimator.

Case ($S_1$) is close to the parametric statistics. In the \iid{} settings, the multipliers $a$, which ensure that $aU_G$ is the unbiased estimator of $\sqrt{V}$, are known for commonly used parametric families: $a=\sqrt{\pi}/2$ for the normal distributions; $a=1$ for the exponential distributions; etc. Therefore, assuming an existence of the superpopulation, we use the same constants for the estimation in the finite population. If a good auxiliary information ${\cal Z}$ is available, then these theoretical $a$ should not so much differ from the corresponding values obtained by case ($S_2$), where $a>0$ is evaluated from $aG=\sqrt{V}$ using ${\cal Z}$ instead of ${\cal X}$. If the scatter of ${\cal X}$ can not be linked to a distributions family, e.g., it is a mixture of two unknown distributions, and there is no other additional information, then we suggest strategy ($S_3$).

\subsection{Numerical analysis}\label{s:simul_out}

We compare efficiencies of strategies ($S_1$) and ($S_2$) in respect of the common estimation by $\sqrt{U_V}$ under presence of outliers. We consider two populations which values of the variable $x$ are generated respectively from two different parametric families: the normal distributions ${\cal N}(\mu, \sigma^2)$, and the gamma distributions ${\cal G}(k, \theta)$ with the shape $k$ and scale $\theta$, where variance is equal to $k\theta^2$. In the case of gamma distribution, the correction $a=k^{-1/2}(2-4I_{0.5}(k+1, k))^{-1}$ depends on $k$, where $I_t(u, v)$ is the regularized incomplete Beta function. For each of these populations, we consecutively increase the part of outliers in the population as follows. Firstly, we select some particular population elements randomly without replacement. Secondly, we replace their values by new generated from the same family of distributions but with different parameters, and we fix these values. In the next steps, the set of outlying elements is increased by selecting from those which still not belong to the outliers.

In particular, the distributions are: ${\cal N}(0, 1)$, and ${\cal N}(0, 9)$ is for generation of outliers; ${\cal G}(3, 1/\sqrt{3})$  (then $a=8\sqrt{3}/15$), and ${\cal G}(3, \sqrt{3})$ is for outliers. We take $N=1000$, $n=200$, and consecutively construct the populations with $p=0, 20, 40, 60, 80, 100$ outliers.

The fixed values of the auxiliary information ${\cal Z}$ are generated by the linear regression $z_i=3+2x_i+\varepsilon_i$, where $\varepsilon_i$, $i=1,\ldots, N$, are \iid{} random variables from ${\cal N}(0, \vartheta^2)$. Since the set ${\cal X}$ is different for different $p$, collections ${\cal Z}$ are different too.

To understand better a role of the auxiliary information in strategy ($S_2$), we simulate different correlations $\rho_{zx}$ between ${\cal Z}$ and ${\cal X}$. The correlation is controlled with the variance $\vartheta^2$ in the linear model. Thus we choose the variance in order to have $\rho_{zx}=0.9, 0.7, 0.5$ approximately. Tables \ref{t1}--\ref{t2} present the comparison of the estimation methods by means of mean square errors and biases.

\begin{table}[!h]
\caption{ ${\cal N}(0, 1)$ with outliers ${\cal N}(0, 9)$. Accuracy by $10\times (BIAS(\cdot), \sqrt{MSE(\cdot)})$.}\label{t1}
 \centering
 \tabcolsep=14pt
 \vspace{1mm}
\scalebox{0.89}{
\begin{tabular}{crrrrr}
\boldhline
    $p/N$       &     $\sqrt{U_V}$  &    \phantom{$\rho_{zx}=0.9$; }($S_1$) &     $\rho_{zx}=0.9$; ($S_2$)   &   $\rho_{zx}=0.7$; ($S_2$) &     $\rho_{zx}=0.5$; ($S_2$)  \\          
\hline
$0.00$  &  $		(-0.01,	0.48)$  &  $		(-0.06,	0.47)$  &  $		(-0.02,	0.47)$  &  $		 (\phantom{-}0.00,	0.47)$  &  $		(-0.05,	0.47)$ \\
$0.02$  &  $		(-0.03,	0.86)$  &  $		(-0.41,	0.73)$  &  $		(-0.13,	0.63)$  &  $		(-0.32,	0.68)$  &  $		(-0.40,	0.72)$ \\
$0.04$  &  $		(-0.03,	0.99)$  &  $		(-0.67,	0.97)$  &  $		(-0.16,	0.75)$  &  $		(-0.44,	0.84)$  &  $		(-0.63,	0.95)$ \\
$0.06$  &  $		(-0.03,	1.10)$  &  $		(-0.92,	1.22)$  &  $		(-0.24,	0.87)$  &  $		(-0.71,	1.07)$  &  $		(-0.86,	1.17)$ \\
$0.08$  &  $		(-0.03,	1.10)$  &  $		(-0.95,	1.26)$  &  $		(-0.28,	0.91)$  &  $		(-0.58,	1.02)$  &  $		(-0.92,	1.23)$ \\
$0.10$  &  $		(-0.03,	1.22)$  &  $		(-1.18,	1.48)$  &  $		(-0.21,	0.98)$  &  $		(-0.91,	1.29)$  &  $		(-1.06,	1.39)$ \\
\boldhline
\end{tabular} 
}

\vspace{0.5mm}

\caption{ ${\cal G}(3, 1/\sqrt{3})$ with outliers ${\cal G}(3, \sqrt{3})$. Accuracy by $10\times (BIAS(\cdot), \sqrt{MSE(\cdot)})$.}\label{t2}
 \centering
 \tabcolsep=14pt
 \vspace{1mm}
\scalebox{0.89}{
\begin{tabular}{crrrrr}
\boldhline
    $p/N$       &     $\sqrt{U_V}$  &    \phantom{$\rho_{zx}=0.9$; }($S_1$) &     $\rho_{zx}=0.9$; ($S_2$)   &   $\rho_{zx}=0.7$; ($S_2$) &     $\rho_{zx}=0.5$; ($S_2$)  \\          
\hline
$0.00$  &  $		(-0.03,	0.71)$  &  $		(-0.18,	0.64)$  &  $		(-0.21,	0.65)$  &  $		(-0.42,	0.73)$  &  $		(-0.60,	0.84)$ \\
$0.02$  &  $		(-0.09,	1.40)$  &  $		(-1.09,	1.40)$  &  $		(-0.48,	1.04)$  &  $		(-1.05,	1.36)$  &  $		(-1.48,	1.70)$ \\
$0.04$  &  $		(-0.08,	1.42)$  &  $		(-1.24,	1.55)$  &  $		(-0.66,	1.18)$  &  $		(-1.21,	1.53)$  &  $		(-1.69,	1.91)$ \\
$0.06$  &  $		(-0.07,	1.47)$  &  $		(-1.56,	1.87)$  &  $		(-0.76,	1.35)$  &  $		(-1.71,	1.99)$  &  $		(-1.98,	2.23)$ \\
$0.08$  &  $		(-0.12,	1.93)$  &  $		(-2.24,	2.56)$  &  $		(-1.08,	1.71)$  &  $		(-2.06,	2.41)$  &  $		(-2.70,	2.94)$ \\
$0.10$  &  $		(-0.13,	1.99)$  &  $		(-2.56,	2.89)$  &  $		(-1.31,	1.97)$  &  $		(-2.54,	2.88)$  &  $		(-2.85,	3.14)$ \\
\boldhline
\end{tabular} 
}
\end{table}

It is seen from Table \ref{t1} that strategy ($S_1$) improves the estimator $\sqrt{U_V}$ where the proportion $p/N$ is smaller. For $p/N$ larger than $0.04$, ($S_1$) becomes inefficient (by $MSE(\cdot)$) because its bias is large, since the fixed correction $a$ is to much approximate for the mix of the normal distributions. Strategy ($S_2$) is the best under strong correlation between $x$ and $z$, because the estimation bias is well-corrected. The efficiency of ($S_2$) decreases with the decrease of the correlation $\rho_{zx}$.

Table \ref{t2} shows similar results for the asymmetric gamma distributions. Here outliers affect the estimators stronger because the distribution of outliers has larger mean (location) in addition. Therefore, strategies ($S_1$) and ($S_2$) are efficient for smaller proportions $p/N$ than in Table \ref{t1}. 

We conclude that strategies ($S_1$) and ($S_2$), and thus the GMD statistic, are efficient, in respect of $\sqrt{U_V}$, if there is a small percent of outliers in the population. Moreover, there is no loss in the efficiency of the strategies if there are no outliers in the population.

\section{Theoretical properties of the statistics}\label{s:2}

\subsection{Hoeffding's decompositions and variances}

The statistic $U=U_n(\mathbb X)=\sum_{1\leq i<j\leq n} h(X_i,X_j)$, where a function $h\colon {\cal X} \times {\cal X} \to \mathbb R$ satisfies $h(x,y)=h(y,x)$, is called $U$-statistic of degree two. For the cases of the GMD statistic $U_G$ and the sample variance $U_V$, we have 
\begin{equation*}
h(X_1,X_2)={n\choose 2}^{-1}|X_1-X_2|
\end{equation*}
and
\begin{equation*}
h(X_1,X_2)={n\choose 2}^{-1}(X_1-X_2)^2/2,
\end{equation*}
respectively. Following \cite{B_2003}, the Hoeffding decomposition of the $U$-statistic is
\begin{equation}\label{U_dec} 
U=\E U+U_1+U_2,
\end{equation}
where $U_1=\sum_{i=1}^n g_1(X_i)$ and $U_2=\sum_{1\leq i<j\leq n} g_2(X_i, X_j)$ are centered and uncorrelated linear and quadratic parts, respectively.  Here, for $1\leq k\leq N$,
\begin{equation*}
g_1(x_k)=(n-1)\frac{N-1}{N-2}\E \left(h(X_1,X_2)-\E h(X_1,X_2) \,\middle|\, X_1=x_k \right)
\end{equation*}
and, for $1\leq k \neq l\leq N$,
\begin{equation*}
g_2(x_k,x_l)=h(x_k,x_l)-\E h(X_1,X_2)-(n-1)^{-1}\left( g_1(x_k)+g_1(x_l) \right).
\end{equation*}
The so-called first- and second-order influence functions $g_1(\cdot)$ and $g_2(\cdot,\cdot)$ have usually a different impact to the variance of $U$-statistic. As in cases of any other linearization techniques, it is expected that the linear part in \eqref{U_dec} dominates against the remainder in the sense of variance size. In particular, we consider structures of the variances of the statistics $U_G$ and $U_V$ by formula (2.6) in \cite{BG_2001}: 
\begin{equation}\label{U_var_dec}
\Var U=\frac{n(N-n)}{N-1}\sigma_1^2+{n\choose 2}{N-n\choose 2}{N-2\choose 2}^{-1}\sigma_2^2,
\end{equation}
where it is denoted $\sigma_1^2=\E g_1^2(X_1)$ and $\sigma_2^2=\E g_2^2(X_1,X_2)$. Let us elaborate the statistics of interest. 
\smallskip

\noindent\textbf{GMD statistic.} To find the influence functions, we rewrite \eqref{GMD} into the alternative form 
\begin{equation*}
U_G={n\choose 2}^{-1}\sum_{j=1}^n (2j-n-1) X_{j:n},
\end{equation*}
where $X_{1:n}\leq\cdots\leq X_{n:n}$ are the order statistics of the observations $\mathbb X$, and apply the Hoeffding decomposition results for $L$-statistics from \cite{C_2012}.  Denote $a_i=(2i-N)/N$, $1\leq i\leq N-1$. Then, for $1\leq k\leq N$,
\begin{equation*}\label{GMD_g1}
g_1(x_k)=-\frac{2}{n}\frac{N}{N-2}\sum_{i=1}^{N-1}\left( \mathbb{I}\{ i\geq k\}-\frac{i}{N}\right)a_i\Delta_i,
\end{equation*} 
where $\mathbb I\{\cdot\}$ is the indicator function, and, for $1\leq k<l\leq N$,
\begin{equation*}\label{GMD_g2}
g_2(x_k,x_l)=-\frac{4}{n(n-1)}\sum_{i=1}^{N-1} \phi_{k,l}(i)\Delta_i,
\end{equation*}
where
\begin{equation*}
\phi_{k,l}(i)=\begin{cases}
i(i-1)/A, &\text{if $1\leq i<k$,}\\
-(i-1)(N-i-1)/A, &\text{if $k\leq i<l$,}\\
(N-i-1)(N-i)/A, &\text{if $l\leq i<N$,}
\end{cases}
\end{equation*}
with $A=(N-1)(N-2)$. Next, direct calculations give the expressions of variance decomposition \eqref{U_var_dec} components:
\begin{equation}\label{GMD_sigma1}
\sigma_1^2=\frac{4}{n^2}\frac{1}{(N-2)^2}\Bigg[ \sum_{i=1}^{N-1}i(N-i)a_i^2\Delta_i^2+2\sum_{1\leq i<j\leq N-1}i(N-j)a_ia_j\Delta_i\Delta_j\Bigg]
\end{equation}
and
\begin{equation}\label{GMD_sigma2}
\begin{split}
\sigma_2^2=\frac{16}{n^2(n-1)^2}\frac{1}{N(N-1)^2(N-2)}\Bigg[& \sum_{i=1}^{N-1}i(i-1)(N-i-1)(N-i)\Delta_i^2 \\
&\quad +2\sum_{1\leq i<j\leq N-1}i(i-1)(N-j-1)(N-j)\Delta_i\Delta_j\Bigg].
\end{split}
\end{equation}
%palyginti su G_1962
\smallskip

\noindent\textbf{Sample variance.} Denote the population moments $b_1=\E X_1$ and $\mu_k=\E(X_1-b_1)^k$, for $k=2,\ldots, 6$. Then, for $1\leq k\leq N$,
\begin{equation}\label{V_g_1}
g_1(x_k)=\frac{1}{n}\frac{N}{N-2}\left[(x_k-b_1)^2-\mu_2\right],
\end{equation}
and, for $1\leq k<l\leq N$,
\begin{equation}\label{V_g_2}
g_2(x_k,x_l)=\frac{1}{n(n-1)}\left\{ (x_k-x_l)^2+\frac{2N}{(N-1)(N-2)}\mu_2-\frac{N}{N-2}\left[ (x_k-b_1)^2+(x_l-b_1)^2\right]\right\}.
\end{equation}
After strightforward calculations, we obtain the following formulas:
\begin{equation}\label{V_sigma_1}
\sigma_1^2=\frac{1}{n^2}\left( \frac{N}{N-2}\right)^2 \left( \mu_4-\mu_2^2\right)
\end{equation}
and
\begin{equation}\label{V_sigma_2}
\sigma_2^2=\frac{4}{n^2(n-1)^2}\frac{N}{(N-1)(N-2)}\left( \frac{N^2-3N+3}{N-1}\mu_2^2-\mu_4\right).
\end{equation}
In fact, various expressions of $\Var U_V$ are known in the literature. For a comparison, we mention just that appeared in \cite{IK_1944}.

\subsection{Asymptotic normality}

Common inferences about statistics are based on knowledge of their distributions. If exact distributions cannot be accessed, then, for samples of a sufficiently large size, the normal approximation to distributions is usually appropriate. Here, for the  statistics under investigation, we give sufficient and simple conditions where the distribution function
\begin{equation}\label{F_nS}
F_{nS}(y)=\PP \left\{ U-\E U\leq yS\right\}
\end{equation}
of the Studentized $U$-statistic is asymptotically normal as the sample size increases. Here 
\begin{equation}\label{U_var_jack}
S^2=S^2(\mathbb X)=\left( 1-\frac{n}{N}\right) \frac{n-1}{n}\sum_{i=1}^n \left( U_{n-1}(\mathbb X\backslash X_i)-\widebar U\right)^2, \quad \text{where} \quad \widebar U=\frac{1}{n}\sum_{i=1}^n  U_{n-1}(\mathbb X\backslash X_i),
\end{equation}
is the jackknife estimator of the variance for any $U$-statistic.

In the finite populations asymptotics, the population size increases together with the sample size. We denote $n_*=\min\{n, N-n\}$, which tends to infinity as $n$ does in the \iid{} setup. Next, to be correct in the formulation of asymptotic results, a sequence of values  ${\cal X}_r=\{ x_{r,1},\ldots, x_{r, N_r}\}$ in the populations, with $N_r\to\infty$ as $r\to\infty$, and a sequence of statistics $U_{n_r}(\mathbb X_r)$, where $\mathbb X_r=\{ X_{r,1},\ldots, X_{r, n_r}\}$ is a sample drawn without replacement from  ${\cal X}_r$, should be considered. Further, we omit the subscript $r$ for these and other quantities for notational simplicity.  

Denote $\tau^2=n(1-n/N)$ for short.  \cite{ER_1959}, and \cite{H_1960} Lindeberg-type condition: for every $\varepsilon>0$,
\begin{equation}\label{lind_2}
\sigma_1^{-2} \E g_1^2(X_1) \mathbb{I}\{ \left| g_1(X_1) \right|>\varepsilon\tau\sigma_1\}=o(1) \quad \text{as} \quad n_* \to \infty,
\end{equation}
imposed on the linear part of $U$-statistic, is necessary for the normality of asymptotically linear statistics as the size $n_*$ grows. This condition, together with moments conditions ensuring the asymptotic linearity, is sufficient for the satistics  $U_G$ and $U_V$ by the following limit theorem. 
\begin{thm}\label{t:norm_u}
Assume that $n_* \to \infty$.  Let \eqref{lind_2} be satisfied. Assume that for all  $n_*$: (i) for $U_G$, $\E X_1^2\leq C_1<\infty$ holds; (ii) for $U_V$, $\E X_1^4\leq C_2<\infty$ holds. Then, for $U_G$ and $U_V$, \eqref{F_nS} tends to the standard normal distribution function $\Phi(y)$ for every $y\in\mathbb R$, respectively.
\end{thm} 
\begin{proof}
To be consistent with conditions imposed on symmetric (and thus $U$-) statistics in \cite{BG_2001}, consider normalized versions of the statistics of interest: $\sqrt{n} U_G$ and $\sqrt{n} U_V$. Then the variances of linear parts from the decompositions of these statistics are bounded away from zero, and are finite if the corresponding conditions \emph{(i)} and \emph{(ii)} are satisfied. Therefore, in the case of $U_G$, the normality proof follows immediately from Theorem 1 in \cite{C_2013b} through Proposition 3 in \cite{BG_2001}. In the case of $U_V$, by Theorem 1 and Proposition 3 in \cite{BG_2001}, it suffices to verify that the variance of quadratic part of  $\sqrt{n}U_V$ tends to zero as $n_* \to \infty$. If \emph{(ii)} is satisfied, it follows easily from the explicit formulas above.
\end{proof}

\subsection{True one-term Edgeworth expansions}

When the sample size is not a large, the normal approximation to \eqref{F_nS} can be inaccurate. Then the one-term Edgeworth expansion
\begin{equation}\label{H_nS}
H_{nS}(y)=\Phi(y)+\frac{\left( 1-2n/N+\left(2-n/N\right)y^2\right)\alpha+3\left(y^2+1\right)\kappa}{6\tau}\varphi(y),
\end{equation}
for Studentized $U$-statistics, constructed in \cite{B_2003}, can be an improvement. Here $\varphi(y)$ is the standard normal density function, and  \begin{equation*}\label{alpha_kappa}
\alpha=\sigma_1^{-3}\E g_1^3(X_1) \quad \text{and} \quad \kappa=\sigma_1^{-3}\tau^2\E g_2(X_1,X_2)g_1(X_1)g_1(X_2)
\end{equation*}
 are the population characteristics. Next, we give detailed expressions of these parameters for both statistics of interest.
\smallskip

\noindent\textbf{GMD statistic.} Routine but tedious combinatorial calculations give
\begin{equation}\label{G_alpha}
\begin{split}
\alpha=-\sigma_1^{-3}\frac{8}{n^3}\frac{1}{(N-2)^3} &\Bigg[ \sum_{i=1}^{N-1}i(N-2i)(N-i)a_i^3\Delta_i^3+3\sum_{1\leq i<j\leq N-1}i(N-2i)(N-j)a_i^2a_j\Delta_i^2\Delta_j\\
&\quad +3\sum_{1\leq i<j\leq N-1}i(N-2j)(N-j)a_ia_j^2\Delta_i\Delta_j^2\\
& \qquad +6\sum_{1\leq i<j<m\leq N-1}i(N-2j)(N-m)a_ia_ja_m\Delta_i\Delta_j\Delta_m\Bigg]
\end{split}
\end{equation}
and
\begin{equation}\label{G_kappa}
\kappa=-\sigma_1^{-3}\tau^2\frac{16}{n^3(n-1)}\frac{N}{(N-1)^2(N-2)^3}\sum_{i=1}^{N-1}\sum_{j=1}^{N-1}\sum_{m=1}^{N-1}c_{ijm}a_ja_m\Delta_i\Delta_j\Delta_m,
\end{equation}
where
\begin{equation*}
c_{ijm}=\begin{cases}
i(i-1)(N-m)[N-j-1+N^{-1}j(m-j)], &\text{if $i\leq j\leq m$,}\\
i(i-1)(N-j)[N-m-1+N^{-1}m(m-j)], &\text{if $i\leq m<j$,}\\
j(N-m)[(i-1)(N-i-1)+N^{-1}\{ (N-i)(N-i-1)(i-j)+i(i-1)(m-i)\}], &\text{if $j<i<m$,}\\
m(N-j)[(i-1)(N-i-1)+N^{-1}\{ i(i-1)(i-j)+(N-i-1)(N-i)(m-i)\}], &\text{if $m<i<j$,}\\
j(N-i-1)(N-i)[m-1+N^{-1}(N-m)(m-j)], &\text{if $j<m\leq i$,}\\
m(N-i-1)(N-i)[j-1+N^{-1}(N-j)(m-j)], &\text{if $m\leq j\leq i$.}
\end{cases}
\end{equation*}
These formulas are new in the literature.
\smallskip

\noindent\textbf{Sample variance.} With strightforward calculations one can arrive to the following results:
\begin{equation}\label{V_alpha}
\alpha=\sigma_1^{-3}\frac{1}{n^3}\left( \frac{N}{N-2}\right)^3 \left( 2\mu_2^3-3\mu_4\mu_2+\mu_6\right)
\end{equation}
and
\begin{equation}\label{V_kappa} 
\kappa=\sigma_1^{-3}\tau^2\frac{2}{n^3(n-1)}\left( \frac{N}{N-2}\right)^3\frac{1}{N-1}\left( -(N-2)\mu_3^2-\frac{2N-1}{N-1}\mu_4\mu_2+\frac{N}{N-1}\mu_2^3+\mu_6\right). 
\end{equation}
Note that \eqref{V_kappa} can be simplified (approximated) by leaving the term with $\mu_3^2$ in the brackets only. For comparison, expressions similar to these can be identified in the Edgeworth approximation given by \cite{KW_1990} for standardized sample variance. 
\smallskip

While an error of the normal approximation is typically of the order $O(n_*^{-1/2})$, see, e.g., \cite{ZC_1990} for the case of standardized $U$-statistics, the error of the true (with known parameters $\alpha$ and $\kappa$) one-term Edgeworth approximation \eqref{H_nS} is of the order $o(n_*^{-1/2})$ under certain conditions. The first condition, from those, is the asymptotical nonlatticeness of the linear part of $U$-statistic: for every $\varepsilon>0$ and every $B>0$, 
\begin{equation}\label{non_lat}
\liminf_{n_*\to\infty}\sup_{\varepsilon<|t|<B} \left| \E \exp\left\{ \mathrm{i}t\sigma_1^{-1}g_1(X_1)\right\}\right|<1,
\end{equation}
see \cite{BG_2001}. This and other specific sufficient conditions for the statistics $U_G$ and $U_V$ are summarized in the following theorem. 
\begin{thm}\label{t:edge_u}
Assume that $n_* \to \infty$ and $(1-n/N)\tau \to \infty$. Let \eqref{non_lat} be satisfied. Assume that, for some $\delta>0$ and for all  $n_*$: (i) for $U_G$, $\E |X_1|^{6+\delta}\leq C_1<\infty$ holds; (ii) for $U_V$, $\E |X_1|^{12+\delta}\leq C_2<\infty$ holds. Then, we have 
\begin{equation*}
\sup_{y\in\mathbb R}|F_{nS}(y)-H_{nS}(y)|=o(n_*^{-1/2}) \quad \text{as} \quad n_* \to \infty,
\end{equation*}
 for $U_G$ and $U_V$, respectively.
\end{thm} 
\begin{proof}
In the case of $U_G$, the proof is the corollary of Theorem 1 in \cite{B_2003} following technique in the proof of Theorem 1 in \cite{C_2012}. In particular, by these theorems, the boundedness of the characteristics $\beta_s=\sigma_1^{-s} \E |g_1(X_1)|^s$ and $\gamma_s=\sigma_1^{-s}\tau^{2s} \E |g_2(X_1,X_2)|^s$, as $n_* \to \infty$, must be verified for $s>6$ only.  

In the case of $U_V$, the task is the same. By \eqref{V_g_1}, for $s\geq 1$, applying inequalities $|a-b|^s\leq 2^{s-1}(a^s+b^s)$ where $a,b\geq 0$, and $\mu_2^s\leq \mu_{2s}$, we get
\begin{equation}\label{est_g1s}
\begin{split}
\E |g_1(X_1)|^s&=\frac{1}{N}\sum_{k=1}^N |g_1(x_k)|^s\leq \frac{2^{s-1}}{n^s}  \left(\frac{N}{N-2} \right)^s  \frac{1}{N}\sum_{k=1}^N \left( (x_k-b_1)^{2s}+\mu_2^s\right) \leq \left(\frac{2N}{n(N-2)} \right)^s \mu_{2s}. 
\end{split}
\end{equation}
By \eqref{V_g_2}, for $1\leq k<l\leq N$, applying $(x_k-x_l)^2\leq 2\left( (x_k-b_1)^2+(x_l-b_1)^2\right)$, we have 
\begin{equation*}
\begin{split}
|g_2(x_k,x_l)|&\leq \frac{1}{n(n-1)} \left(\frac{3N-4}{N-2}\left( (x_k-b_1)^2+(x_l-b_1)^2\right)+\frac{2N}{(N-1)(N-2)}\mu_2 \right) \\
&\leq \frac{3}{n(n-1)} \frac{N}{N-2} \left((x_k-b_1)^2+(x_l-b_1)^2+\mu_2 \right).
\end{split}
\end{equation*}
Then, for $s\geq 1$, similarly as in \eqref{est_g1s}, applying $(a+b)^s\leq 2^{s-1}(a^s+b^s)$ twice, where $a,b\geq 0$, and noting that $\sum_{1\leq k<l\leq N} \left( (x_k-b_1)^{2s}+(x_l-b_1)^{2s}\right)=N(N-1)\mu_{2s}$, we obtain
\begin{equation}\label{est_g2s}
\begin{split}
\E |g_2(X_1,X_2)|^s&={N\choose 2}^{-1}\sum_{1\leq k<l\leq N} |g_2(x_k,x_l)|^s \\
&\leq \frac{3^s}{n^s(n-1)^s} \left(\frac{N}{N-2} \right)^s{N\choose 2}^{-1}\sum_{1\leq k<l\leq N} \left((x_k-b_1)^2+(x_l-b_1)^2+\mu_2 \right)^s \\
&\leq  \frac{3^s2^{s-1}}{n^s(n-1)^s} \left(\frac{N}{N-2} \right)^s{N\choose 2}^{-1}\sum_{1\leq k<l\leq N} \left(2^{s-1}\left((x_k-b_1)^{2s}+(x_l-b_1)^{2s}\right)+\mu_2^s \right) \\
&\leq  \frac{3^s2^{s-1}(2^s+1)}{n^s(n-1)^s} \left(\frac{N}{N-2} \right)^s\mu_{2s}.
\end{split}
\end{equation}
Then we get from \eqref{est_g1s}, \eqref{est_g2s} and \eqref{V_sigma_1} that
\begin{equation*}
\beta_s\leq \frac{2^s\mu_{2s}}{\left( \mu_4-\mu_2^2\right)^{s/2}} \quad \text{and} \quad  \gamma_s\leq 3^s 2^{2s-1}\left( 2^s+1\right)\left( 1-\frac{n}{N}\right)^s \frac{\mu_{2s}}{\left( \mu_4-\mu_2^2\right)^{s/2}}.
\end{equation*}
The proof is completed.
\end{proof}

\section{Estimation of parameters}\label{s:4}

\subsection{Estimators of variances}

The jackknife variance estimator, defined by \eqref{U_var_jack}, is universal for $U$- and other statistics but it is not the best for the particular ones. In \cite{PC_2013}, bootstrap and calibrated estimators, constructed for general $L$-statistics, are comparatively complex. Here, for both statistics of interest, we have explicit expressions of their variances. Therefore, more natural as well as simple estimators of the variances are possible. We give here, in fact, plug-in estimators of the variances, replacing population moments by their empirical counterparts in the parameters $\sigma_1^2$ and $\sigma_2^2$ defining variance \eqref{U_var_dec}.
\smallskip

\noindent\textbf{GMD statistic.} Denote $\Delta_{i:n}=X_{i+1:n}-X_{i:n}$ and $A_i=(2i-n)/n$, for $1\leq i\leq n-1$. Then the estimators of the variance components \eqref{GMD_sigma1} and \eqref{GMD_sigma2} are  
\begin{equation}\label{GMD_sigma1_est}
\hat\sigma_{1G}^2=\frac{4}{n^4}\left(\frac{N}{N-2}\right)^2\Bigg[ \sum_{i=1}^{n-1}i(n-i)A_i^2\Delta_{i:n}^2+2\sum_{1\leq i<j\leq n-1}i(n-j)A_iA_j\Delta_{i:n}\Delta_{j:n}\Bigg]
\end{equation}
and
\begin{equation}\label{GMD_sigma2_est}
\begin{split}
\hat\sigma_{2G}^2=\frac{16}{n^4(n-1)^4}\frac{N}{N-2}\Bigg[& \sum_{i=1}^{n-1}i(i-1)(n-i-1)(n-i)\Delta_{i:n}^2 \\
& \quad +2\sum_{1\leq i<j\leq n-1}i(i-1)(n-j-1)(n-j)\Delta_{i:n}\Delta_{j:n}\Bigg].
\end{split}
\end{equation}
Denote by $\hat\sigma_G^2$ the estimator of the variance of $U_G$ obtained by plugging \eqref{GMD_sigma1_est} and \eqref{GMD_sigma2_est} into \eqref{U_var_dec}.
\smallskip

\noindent\textbf{Sample variance.} Denote the sample moments by $m_k=n^{-1}\sum_{i=1}^n(X_i-n^{-1}\sum_{j=1}^n X_j)^k$,  for $k=2,\ldots, 6$. Replacing the central moments in \eqref{V_sigma_1} and \eqref{V_sigma_2} by the corresponding empirical moments, we get 
\begin{equation}\label{V_sigma_1_est}
\hat\sigma_{1V}^2=\frac{1}{n^2}\left( \frac{N}{N-2}\right)^2 \left( m_4-m_2^2\right)
\end{equation}
and
\begin{equation}\label{V_sigma_2_est}
\hat\sigma_{2V}^2=\frac{4}{n^2(n-1)^2}\frac{N}{(N-1)(N-2)}\left( \frac{N^2-3N+3}{N-1}m_2^2-m_4\right).
\end{equation}
Let $\hat\sigma_V^2$ denote the estimator of the variance of $U_V$ obtained by plugging \eqref{V_sigma_1_est} and \eqref{V_sigma_2_est} into \eqref{U_var_dec}.

\subsection{Estimators for parameters defining Edgeworth expansions}\label{s:param_est}

In order to apply the one-term Edgeworth approximation \eqref{H_nS} to the distribution functions of the statistics, the parameters $\alpha$ and $\kappa$ must be evaluated. Firstly, case (A), analogously to the variance estimation case, we construct estimators of the parameters directly from the explicit expressions available. Secondly, case (B), we assume that  the auxiliary variable $z$ is at our disposal with the known values $\{ z_1,\ldots, z_N\}$ for all population elements. It is expected in this case, that $z$ is well-correlated with the study variable $x$. Then the estimators below are immediately obtained from the true values of the parameters.  
\smallskip

\noindent\textbf{GMD statistic.} Case (A). With the notations used for the variance estimator, by formulas \eqref{G_alpha} and \eqref{G_kappa}, the estimators are
\begin{equation}\label{G_alpha_est}
\begin{split}
\hat\alpha_G=-\hat\sigma_{1G}^{-3}\frac{8}{n^6}\left(\frac{N}{N-2}\right)^3 &\Bigg[ \sum_{i=1}^{n-1}i(n-2i)(n-i)A_i^3\Delta_{i:n}^3+3\sum_{1\leq i<j\leq n-1}i(n-2i)(n-j)A_i^2A_j\Delta_{i:n}^2\Delta_{j:n}\\
& \quad +3\sum_{1\leq i<j\leq n-1}i(n-2j)(n-j)A_iA_j^2\Delta_{i:n}\Delta_{j:n}^2\\
& \qquad +6\sum_{1\leq i<j<m\leq n-1}i(n-2j)(n-m)A_iA_jA_m\Delta_{i:n}\Delta_{j:n}\Delta_{m:n}\Bigg]
\end{split}
\end{equation}
and
\begin{equation}\label{G_kappa_est}
\hat\kappa_G=-\hat\sigma_{1G}^{-3}\tau^2\frac{16}{n^5(n-1)^3}\left(\frac{N}{N-2}\right)^3 \,\sum_{i=1}^{n-1}\sum_{j=1}^{n-1}\sum_{m=1}^{n-1}C_{ijm}A_jA_m\Delta_{i:n}\Delta_{j:n}\Delta_{m:n},
\end{equation}
with the case function
\begin{equation*}
C_{ijm}=\begin{cases}
i(i-1)(n-m)[n-j-1+n^{-1}j(m-j)], &\text{if $i\leq j\leq m$,}\\
i(i-1)(n-j)[n-m-1+n^{-1}m(m-j)], &\text{if $i\leq m<j$,}\\
j(n-m)[(i-1)(n-i-1)+n^{-1}\{ (n-i)(n-i-1)(i-j)+i(i-1)(m-i)\}], &\text{if $j<i<m$,}\\
m(n-j)[(i-1)(n-i-1)+n^{-1}\{ i(i-1)(i-j)+(n-i-1)(n-i)(m-i)\}], &\text{if $m<i<j$,}\\
j(n-i-1)(n-i)[m-1+n^{-1}(n-m)(m-j)], &\text{if $j<m\leq i$,}\\
m(n-i-1)(n-i)[j-1+n^{-1}(n-j)(m-j)], &\text{if $m\leq j\leq i$.}
\end{cases}
\end{equation*}
Case (B). Having the additional information, the ordered sequence of the values $z_1,\ldots, z_N$ is used instead of $x_1\leq\cdots\leq x_N$ in the expressions \eqref{G_alpha} and \eqref{G_kappa} of the true parameters $\alpha$ and $\kappa$. Denote the resulting estimates by ${}_z\hat\alpha_G$ and ${}_z\hat\kappa_G$.  
\smallskip

\noindent\textbf{Sample variance.} Case (A). From population parameters \eqref{V_alpha} and \eqref{V_kappa}, we have the following plug-in estimators:
\begin{equation}\label{V_alpha_est}
\hat\alpha_V=\hat\sigma_{1V}^{-3}\frac{1}{n^3}\left( \frac{N}{N-2}\right)^3 \left( 2m_2^3-3m_4m_2+m_6\right)
\end{equation}
and
\begin{equation}\label{V_kappa_est} 
\hat\kappa_V=\hat\sigma_{1V}^{-3}\tau^2\frac{2}{n^3(n-1)}\left( \frac{N}{N-2}\right)^3\frac{1}{N-1}\left( -(N-2)m_3^2-\frac{2N-1}{N-1}m_4m_2+\frac{N}{N-1}m_2^3+m_6\right). 
\end{equation}
Case (B). In \eqref{V_alpha} and \eqref{V_kappa}, the central population moments $\mu_k$ are evaluated using the values $z_1,\ldots, z_N$. Then denote the new estimates by ${}_z\hat\alpha_V$ and ${}_z\hat\kappa_V$.

\section{Empirical Edgeworth and bootstrap approximations}\label{s:5}

Replacing the population parameters $\alpha$ and $\kappa$ in Edgeworth expansion \eqref{H_nS} by their estimators, we obtain the so-called empirical Edgeworth expansion. If the particular estimators of the parameters are asymptotically consistent, then, under the conditions of Theorem \ref{t:edge_u}, the empirical Edgeworth expansion approximates distribution function \eqref{F_nS} with an error of the same order but in probability. In \cite{B_2001}, consistent jackknife estimators of the parameters were constructed. Bootstrap and calibrated estimators of the parameters were considered in \cite{C_2013a}, and \cite{PC_2013}, respectively. Here, for each of the statistics $U_G$ and $U_V$, we have two new versions of the empirical Edgeworth expansion.
\smallskip

\noindent\textbf{GMD statistic.} By the results in Section \ref{s:param_est}, we have the empirical Edgeworth expansion
\begin{equation}\label{H_nSG}
\widehat H_{nSG}(y)=\Phi(y)+\frac{\left( 1-2n/N+\left(2-n/N\right)y^2\right)\hat\alpha_G+3\left(y^2+1\right)\hat\kappa_G}{6\tau}\varphi(y),
\end{equation}
and, in the case where the auxiliary information is available, the approximation is
\begin{equation}\label{zH_nSG}
{}_z\widehat H_{nSG}(y)=\Phi(y)+\frac{\left( 1-2n/N+\left(2-n/N\right)y^2\right){}_z\hat\alpha_G+3\left(y^2+1\right){}_z\hat\kappa_G}{6\tau}\varphi(y),
\end{equation} 
which is not a random function because the values of the variable $z$ are treated as fixed in the population.  
\smallskip

\noindent\textbf{Sample variance.} The corresponding approximations to the distribution function of the Studentized sample variance are
\begin{equation}\label{H_nSV}
\widehat H_{nSV}(y)=\Phi(y)+\frac{\left( 1-2n/N+\left(2-n/N\right)y^2\right)\hat\alpha_V+3\left(y^2+1\right)\hat\kappa_V}{6\tau}\varphi(y),
\end{equation}
and
\begin{equation}\label{zH_nSV}
{}_z\widehat H_{nSV}(y)=\Phi(y)+\frac{\left( 1-2n/N+\left(2-n/N\right)y^2\right){}_z\hat\alpha_V+3\left(y^2+1\right){}_z\hat\kappa_V}{6\tau}\varphi(y),
\end{equation} 
where the later does not depend on the sample.
\smallskip

Estimators of the parameters $\alpha$ and $\kappa$ in expansion \eqref{H_nSV} are asymptotically consistent under conditions of Theorem \ref{t:edge_u}. Efficiency of the other empirical Edgeworth expansions is examined in the simulation study in Section \ref{s:6}.
%In the study, we are interested also in a bias of constant functions \eqref{zH_nSG} and \eqref{zH_nSV}. 

It is known that, in general, non-parametric bootstrap approximations to distributions of statistics are usually of a similar accuracy as one-term Edgeworth expansions. We consider here the finite-population bootstrap scheme introduced in \cite{BBH_1994}. We apply the results of \cite{B_2007} where the accuracy of this bootstrap method is considered for $U$-statistics. 

The bootstrap approximation to distribution \eqref{F_nS} is constructed as follows. Write $N=kn+l$, where $0\leq l<n$. Then, given the sample $\mathbb X$, the empirical population $\widetilde{\cal X}$ of size $N$ is formed by taking $k$ copies of $\mathbb X$ and, if $l>0$, adding the remaining $l$ values which are the simple random sample ${\mathbb Y}=\{Y_1, \ldots, Y_l \}$ drawn without replacement from the set $\mathbb X$. With this particular bootstrap population $\widetilde{\cal X}$, one can turn already to an estimator of \eqref{F_nS}, despite that it is only the one of ${n\choose l}$ empirical populations. Next, we draw the simple random sample $\widetilde{\mathbb X}=\{\widetilde X_1,\ldots, \widetilde X_n \}$ without replacement from $\widetilde{\cal X}$. Denote by $\widetilde U=U_n(\widetilde{\mathbb X})$ the bootstrap estimator for the statistic of interest, and introduce the corresponding jackknife estimator $\widetilde S^2=S^2(\widetilde{\mathbb X})$ of the variance of $\widetilde U$ under given population $\widetilde{\cal X}$. Then the bootstrap approximation to \eqref{F_nS} is
\begin{equation}\label{BF_nS}
\widetilde F_{nS}(y)=\PP \{ \widetilde U-\E (\widetilde U \,\, | \,\, \mathbb X, \mathbb Y)\leq y\widetilde S \,\, | \,\, \mathbb X \},
\end{equation}
which averages over all possible empirical populations. The following theorem is on the validity of this approximation for the statistics $U_G$ and $U_V$.
\begin{thm}\label{t:boot_u}
Assume that the conditions of Theorem \ref{t:edge_u} are satisfied. Then, we have 
\begin{equation*}
\sup_{y\in\mathbb R}|F_{nS}(y)-\widetilde F_{nS}(y)|=o_P(n_*^{-1/2}) \quad \text{as} \quad n_* \to \infty,
\end{equation*}
for $U_G$ and $U_V$, respectively.
\end{thm} 
\begin{proof}
It follows from condition (8) in \cite{B_2007}, that it suffices to verify that, for the statistics $U_G$ and $U_V$, the moments $\E (X_1-X_2)^6$ and $\E (X_1-X_2)^{12}$ are bounded for all $n_*$, respectively. By the conditions of theorem, this requirement holds. 
\end{proof}

Denote by $\widetilde F_{nSG}(y)$ and $\widetilde F_{nSV}(y)$ the bootstrap approximations for the statistics $U_G$ and $U_V$, respectively.

\section{Numerical modeling}\label{s:6}

In this section, we illustrate the theoretical results on the second-order approximations to distribution functions of the Studentized GMD statistic and the Studentized sample variance by numerical examples, according to the data framework in Section \ref{s:simul_out}. Thus we consider also how outliers affect these approximations.

For the statistics $U_G$ and $U_V$, denote their `exact' distribution functions by $F_{nSG}(y)$ and $F_{nSV}(y)$, respectively. In the simulation experiments, these functions were evaluated by the Monte--Carlo method, drawing independently $10^6$ samples without replacement from the population and using all values ${\cal X}$, as well as their bootstrap approximations based on the one (because of $N=kn$) empirical population $\widetilde{\cal X}$ constructed from the particular sample $\mathbb X$. Denote true Edgeworth approximations \eqref{H_nS} of the statistics by $H_{nSG}(y)$ and $H_{nSV}(y)$, respectively. To measure an efficiency of the empirical Edgeworth approximations $\widehat H_{nSG}(y)$ and $\widehat H_{nSV}(y)$, and the bootstrap approximations $\widetilde F_{nSG}(y)$ and $\widetilde F_{nSV}(y)$, $10^3$ samples without replacement were drawn independently from the population. 

More specifically, in the tables below, the `exact' distribution functions of the statistics, their normal approximation, the true one-term Edgeworth expansions, the corresponding estimated Edgeworth approximations of two types, and the bootstrap approximations are represented by the several commonly used $q$-quantiles, $q=0.01, 0.05, 0.10, 0.90, 0.95, 0.99$. For the approximations, with the quantiles dependent on the sample, we give two characteristics of the efficiency: the empirical expectations $\hE(\cdot)$ and standard errors $\hS(\cdot)$ from the realizations of these quantiles.   

Tables \ref{t3}--\ref{t6} present results of the approximations, where there are no outliers (the case of $p/N=0$) in the same underlying populations generated from the normal and gamma distribution in Section \ref{s:simul_out}. The correlation is $\rho_{zx}=0.7$.

\begin{table}[!h]
\caption{Approximations to $F_{nSG}(y)$ under ${\cal N}(0, 1)$ with $0\%$ outliers from ${\cal N}(0, 9)$, and $\rho_{zx}=0.7$.}\label{t3}
 \centering
 \tabcolsep=14pt
 \vspace{1mm}
\scalebox{0.89}{
\begin{tabular}{rrrrrrr}
\boldhline
$q=$  &  $0.01$  &  $0.05$  &  $0.10$  &  $0.90$  &  $0.95$  &  $0.99$  \\          
\hline
$F_{nSG}^{-1}(q)\approx$                        &  $-2.592$  &  $-1.779$  &  $-1.363$  &  $1.223$  &  $1.546$  &  $2.157$ \\
$\Phi^{-1}(q)\approx$                                &  $-2.326$  &  $-1.645$  &  $-1.282$  &  $1.282$  &  $1.645$  &  $2.326$ \\
$H_{nSG}^{-1}(q)\approx$                        &  $-2.528$  &  $-1.762$  &  $-1.357$  &  $1.214$  &  $1.536$  &  $2.096$ \\
${}_z\widehat H_{nSG}^{-1}(q)\approx$  &  $-2.513$  &  $-1.752$  &  $-1.350$  &  $1.220$  &  $1.545$  &  $2.116$ \\
$\hE\widehat H_{nSG}^{-1}(q)\approx$    &  $-2.519$  &  $-1.756$  &  $-1.353$  &  $1.217$  &  $1.541$  &  $2.107$ \\
$\hS\widehat H_{nSG}^{-1}(q)\approx$    &  $0.034$  &  $0.023$  &  $0.016$  &  $0.013$  &  $0.021$  &  $0.044$ \\
$\hE\widetilde F_{nSG}^{-1}(q)\approx$   &  $-2.600$  &  $-1.776$  &  $-1.360$  &  $1.222$  &  $1.555$  &  $2.167$ \\
$\hS\widetilde F_{nSG}^{-1}(q)\approx$   &  $0.075$  &  $0.038$  &  $0.027$  &  $\phantom{-}0.020$  &  $\phantom{-}0.026$  &  $\phantom{-}0.044$ \\
\boldhline
\end{tabular} 
}

\vspace{0.5mm}

\caption{Approximations to $F_{nSV}(y)$ under ${\cal N}(0, 1)$ with $0\%$ outliers from ${\cal N}(0, 9)$, and $\rho_{zx}=0.7$.}\label{t4}
 \centering
 \tabcolsep=14pt
 \vspace{1mm}
\scalebox{0.89}{
\begin{tabular}{rrrrrrr}
\boldhline
$q=$  &  $0.01$  &  $0.05$  &  $0.10$  &  $0.90$  &  $0.95$  &  $0.99$  \\          
\hline
$F_{nSV}^{-1}(q)\approx$                         &  $-2.918$  &  $-1.962$  &  $-1.477$  &  $1.160$  &  $1.461$  &  $2.008$ \\
$\Phi^{-1}(q)\approx$                                &  $-2.326$  &  $-1.645$  &  $-1.282$  &  $1.282$  &  $1.645$  &  $2.326$ \\
$H_{nSV}^{-1}(q)\approx$                        &  $-2.680$   &  $-1.882$  &  $-1.447$  &  $1.145$  &  $1.432$  &  $1.878$ \\
${}_z\widehat H_{nSV}^{-1}(q)\approx$  &  $-2.658$  &  $-1.864$  &  $-1.433$  &  $1.155$  &  $1.447$  &  $1.910$ \\
$\hE\widehat H_{nSV}^{-1}(q)\approx$    &  $-2.653$  &  $-1.861$  &  $-1.431$  &  $1.157$  &  $1.450$  &  $1.917$ \\
$\hS\widehat H_{nSV}^{-1}(q)\approx$    &  $0.046$  &  $0.038$  &  $0.029$  &  $0.020$  &  $0.032$  &  $0.068$ \\
$\hE\widetilde F_{nSV}^{-1}(q)\approx$   &  $-2.914$  &  $-1.932$  &  $-1.460$  &  $1.166$  &  $1.473$  &  $2.027$ \\
$\hS\widetilde F_{nSV}^{-1}(q)\approx$  &  $0.147$  &  $0.074$  &  $0.046$  &  $\phantom{-}0.023$  &  $\phantom{-}0.031$  &  $\phantom{-}0.050$ \\
\boldhline
\end{tabular} 
}

\vspace{0.5mm}

\caption{Approximations to $F_{nSG}(y)$ under ${\cal G}(3, 1/\sqrt{3})$ with $0\%$ outliers from ${\cal G}(3, \sqrt{3})$, and $\rho_{zx}=0.7$.}\label{t5}
 \centering
 \tabcolsep=14pt
 \vspace{1mm}
\scalebox{0.89}{
\begin{tabular}{rrrrrrr}
\boldhline
$q=$  &  $0.01$  &  $0.05$  &  $0.10$  &  $0.90$  &  $0.95$  &  $0.99$  \\          
\hline
$F_{nSG}^{-1}(q)\approx$                        &  $-2.888$  &  $-1.903$  &  $-1.443$  &  $1.188$  &  $1.503$  &  $2.062$ \\
$\Phi^{-1}(q)\approx$                                &  $-2.326$  &  $-1.645$  &  $-1.282$  &  $1.282$  &  $1.645$  &  $2.326$ \\
$H_{nSG}^{-1}(q)\approx$                        &  $-2.638$  &  $-1.843$  &  $-1.413$  &  $1.172$  &  $1.468$  &  $1.946$ \\
${}_z\widehat H_{nSG}^{-1}(q)\approx$  &  $-2.572$  &  $-1.793$  &  $-1.378$  &  $1.198$  &  $1.510$  &  $2.038$ \\
$\hE\widehat H_{nSG}^{-1}(q)\approx$    &  $-2.624$  &  $-1.833$  &  $-1.407$  &  $1.177$  &  $1.476$  &  $1.965$ \\
$\hS\widehat H_{nSG}^{-1}(q)\approx$    &  $0.055$   &  $0.045$   &  $0.033$  &  $0.024$  &  $0.037$  &  $0.079$ \\
$\hE\widetilde F_{nSG}^{-1}(q)\approx$   &  $-2.864$  &  $-1.899$  &  $-1.436$  &  $1.190$  &  $1.506$  &  $2.079$ \\
$\hS\widetilde F_{nSG}^{-1}(q)\approx$   &  $0.156$   &  $0.077$   &  $0.048$  &  $\phantom{-}0.025$  &  $\phantom{-}0.035$  &  $\phantom{-}0.060$ \\
\boldhline
\end{tabular} 
}

\vspace{0.5mm}

\caption{Approximations to $F_{nSV}(y)$ under ${\cal G}(3, 1/\sqrt{3})$ with $0\%$ outliers from ${\cal G}(3, \sqrt{3})$, and $\rho_{zx}=0.7$.}\label{t6}
 \centering
 \tabcolsep=14pt
 \vspace{1mm}
\scalebox{0.89}{
\begin{tabular}{rrrrrrr}
\boldhline
$q=$  &  $0.01$  &  $0.05$  &  $0.10$  &  $0.90$  &  $0.95$  &  $0.99$  \\          
\hline
$F_{nSV}^{-1}(q)\approx$                         &  $-3.744$  &  $-2.310$  &  $-1.699$  &  $1.109$  &  $1.391$  &  $1.876$ \\
$\Phi^{-1}(q)\approx$                                &  $-2.326$  &  $-1.645$  &  $-1.282$  &  $1.282$  &  $1.645$  &  $2.326$ \\
$H_{nSV}^{-1}(q)\approx$                        &  $-2.829$   &  $-2.025$  &  $-1.561$  &  $1.077$  &  $1.324$  &  $1.656$ \\
${}_z\widehat H_{nSV}^{-1}(q)\approx$  &  $-2.796$  &  $-1.991$  &  $-1.533$  &  $1.091$  &  $1.347$  &  $1.704$ \\
$\hE\widehat H_{nSV}^{-1}(q)\approx$    &  $-2.788$  &  $-1.985$  &  $-1.529$  &  $1.097$  &  $1.355$  &  $1.719$ \\
$\hS\widehat H_{nSV}^{-1}(q)\approx$    &  $0.075$  &  $0.077$  &  $0.067$  &  $0.040$  &  $0.060$  &  $0.114$ \\
$\hE\widetilde F_{nSV}^{-1}(q)\approx$   &  $-3.642$  &  $-2.267$  &  $-1.655$  &  $1.120$  &  $1.406$  &  $1.909$ \\
$\hS\widetilde F_{nSV}^{-1}(q)\approx$  &  $0.487$  &  $0.277$  &  $0.166$  &  $\phantom{-}0.033$  &  $\phantom{-}0.047$  &  $\phantom{-}0.082$ \\
\boldhline
\end{tabular} 
}
\end{table}

By Table \ref{t3}, the true Edgeworth approximation $H_{nSG}(y)$ improves substantially the normal approximation to $F_{nSG}(y)$. With the help of the auxiliary information, $H_{nSG}(y)$ is estimated well by ${}_z\widehat H_{nSG}(y)$. The bias of this estimate is small in comparison to a possible error of the estimator $\widehat H_{nSG}(y)$. But the later improves the normal approximation to the distribution of $U_G$ too. Differently from all other, the bootstrap approximation $\widetilde F_{nSG}(y)$ is almost unbiased, but its empirical quantiles have larger standard errors compared to the empirical Edgeworth approximation. In Table \ref{t4}, tendencies of the approximations to the distribution function of $U_V$ are the same. In Tables \ref{t5}--\ref{t6}, for the population from the gamma distribution, the results are analogous to those in Tables \ref{t3}--\ref{t4}, but all the corresponding approximations are less accurate. This is because of an asymmetry of the gamma distribution.

\begin{table}[!h]
\caption{Approximations to $F_{nSG}(y)$ under ${\cal N}(0, 1)$ with $6\%$ outliers from ${\cal N}(0, 9)$, and $\rho_{zx}=0.7$.}\label{t7}
 \centering
 \tabcolsep=14pt
 \vspace{1mm}
\scalebox{0.89}{
\begin{tabular}{rrrrrrr}
\boldhline
$q=$  &  $0.01$  &  $0.05$  &  $0.10$  &  $0.90$  &  $0.95$  &  $0.99$  \\          
\hline
$F_{nSG}^{-1}(q)\approx$                        &  $-3.133$  &  $-2.061$  &  $-1.553$  &  $1.143$  &  $1.434$  &  $1.953$ \\
$\Phi^{-1}(q)\approx$                                &  $-2.326$  &  $-1.645$  &  $-1.282$  &  $1.282$  &  $1.645$  &  $2.326$ \\
$H_{nSG}^{-1}(q)\approx$                        &  $-2.745$  &  $-1.940$  &  $-1.490$  &  $1.118$  &  $1.388$  &  $1.783$ \\
${}_z\widehat H_{nSG}^{-1}(q)\approx$  &  $-2.602$  &  $-1.816$  &  $-1.395$  &  $1.184$  &  $1.489$  &  $1.996$ \\
$\hE\widehat H_{nSG}^{-1}(q)\approx$    &  $-2.713$  &  $-1.913$  &  $-1.470$  &  $1.132$  &  $1.410$  &  $1.831$ \\
$\hS\widehat H_{nSG}^{-1}(q)\approx$    &  $0.069$  &  $0.063$  &  $0.050$  &  $0.033$  &  $0.051$  &  $0.103$ \\
$\hE\widetilde F_{nSG}^{-1}(q)\approx$   &  $-3.124$  &  $-2.039$  &  $-1.524$  &  $1.152$  &  $1.449$  &  $1.982$ \\
$\hS\widetilde F_{nSG}^{-1}(q)\approx$   &  $0.267$  &  $0.139$  &  $0.087$  &  $\phantom{-}0.030$  &  $\phantom{-}0.043$  &  $\phantom{-}0.076$ \\
\boldhline
\end{tabular} 
}

\vspace{0.5mm}

\caption{Approximations to $F_{nSV}(y)$ under ${\cal N}(0, 1)$ with $6\%$ outliers from ${\cal N}(0, 9)$, and $\rho_{zx}=0.7$.}\label{t8}
 \centering
 \tabcolsep=14pt
 \vspace{1mm}
\scalebox{0.89}{
\begin{tabular}{rrrrrrr}
\boldhline
$q=$  &  $0.01$  &  $0.05$  &  $0.10$  &  $0.90$  &  $0.95$  &  $0.99$  \\          
\hline
$F_{nSV}^{-1}(q)\approx$                         &  $-4.724$  &  $-2.924$  &  $-2.100$  &  $1.037$  &  $1.280$  &  $1.699$ \\
$\Phi^{-1}(q)\approx$                                &  $-2.326$  &  $-1.645$  &  $-1.282$  &  $1.282$  &  $1.645$  &  $2.326$ \\
$H_{nSV}^{-1}(q)\approx$                        &  $-2.985$   &  $-2.207$  &  $-1.740$  &  $0.979$  &  $1.181$  &  $1.415$ \\
${}_z\widehat H_{nSV}^{-1}(q)\approx$  &  $-2.861$  &  $-2.062$  &  $-1.596$  &  $1.054$  &  $1.291$  &  $1.602$ \\
$\hE\widehat H_{nSV}^{-1}(q)\approx$    &  $-2.889$  &  $-2.097$  &  $-1.634$  &  $1.036$  &  $1.266$  &  $1.560$ \\
$\hS\widehat H_{nSV}^{-1}(q)\approx$    &  $0.086$  &  $0.098$  &  $0.095$  &  $0.050$  &  $0.075$  &  $0.130$ \\
$\hE\widetilde F_{nSV}^{-1}(q)\approx$   &  $-4.600$  &  $-2.792$  &  $-1.974$  &  $1.069$  &  $1.333$  &  $1.787$ \\
$\hS\widetilde F_{nSV}^{-1}(q)\approx$  &  $1.064$  &  $0.663$  &  $0.430$  &  $\phantom{-}0.042$  &  $\phantom{-}0.057$  &  $\phantom{-}0.095$ \\
\boldhline
\end{tabular} 
}

\vspace{0.5mm}

\caption{Approximations to $F_{nSG}(y)$ under ${\cal G}(3, 1/\sqrt{3})$ with $6\%$ outliers from ${\cal G}(3, \sqrt{3})$, and $\rho_{zx}=0.7$.}\label{t9}
 \centering
 \tabcolsep=14pt
 \vspace{1mm}
\scalebox{0.89}{
\begin{tabular}{rrrrrrr}
\boldhline
$q=$  &  $0.01$  &  $0.05$  &  $0.10$  &  $0.90$  &  $0.95$  &  $0.99$  \\          
\hline
$F_{nSG}^{-1}(q)\approx$                        &  $-3.224$  &  $-2.068$  &  $-1.546$  &  $1.148$  &  $1.445$  &  $1.966$ \\
$\Phi^{-1}(q)\approx$                                &  $-2.326$  &  $-1.645$  &  $-1.282$  &  $1.282$  &  $1.645$  &  $2.326$ \\
$H_{nSG}^{-1}(q)\approx$                        &  $-2.740$  &  $-1.933$  &  $-1.483$  &  $1.124$  &  $1.395$  &  $1.795$ \\
${}_z\widehat H_{nSG}^{-1}(q)\approx$  &  $-2.601$  &  $-1.815$  &  $-1.394$  &  $1.186$  &  $1.491$  &  $1.997$ \\
$\hE\widehat H_{nSG}^{-1}(q)\approx$    &  $-2.717$  &  $-1.915$  &  $-1.470$  &  $1.134$  &  $1.410$  &  $1.827$ \\
$\hS\widehat H_{nSG}^{-1}(q)\approx$    &  $0.064$   &  $0.060$   &  $0.049$  &  $0.032$  &  $0.048$  &  $0.097$ \\
$\hE\widetilde F_{nSG}^{-1}(q)\approx$   &  $-3.213$  &  $-2.057$  &  $-1.531$  &  $1.154$  &  $1.453$  &  $1.990$ \\
$\hS\widetilde F_{nSG}^{-1}(q)\approx$   &  $0.268$   &  $0.133$   &  $0.083$  &  $\phantom{-}0.029$  &  $\phantom{-}0.043$  &  $\phantom{-}0.075$ \\
\boldhline
\end{tabular} 
}

\vspace{0.5mm}

\caption{Approximations to $F_{nSV}(y)$ under ${\cal G}(3, 1/\sqrt{3})$ with $6\%$ outliers from ${\cal G}(3, \sqrt{3})$, and $\rho_{zx}=0.7$.}\label{t10}
 \centering
 \tabcolsep=14pt
 \vspace{1mm}
\scalebox{0.89}{
\begin{tabular}{rrrrrrr}
\boldhline
$q=$  &  $0.01$  &  $0.05$  &  $0.10$  &  $0.90$  &  $0.95$  &  $0.99$  \\          
\hline
$F_{nSV}^{-1}(q)\approx$                         &  $-4.895$  &  $-2.890$  &  $-2.042$  &  $1.045$  &  $1.296$  &  $1.725$ \\
$\Phi^{-1}(q)\approx$                                &  $-2.326$  &  $-1.645$  &  $-1.282$  &  $1.282$  &  $1.645$  &  $2.326$ \\
$H_{nSV}^{-1}(q)\approx$                        &  $-2.978$   &  $-2.198$  &  $-1.728$  &  $0.988$  &  $1.193$  &  $1.430$ \\
${}_z\widehat H_{nSV}^{-1}(q)\approx$  &  $-2.864$  &  $-2.064$  &  $-1.597$  &  $1.055$  &  $1.292$  &  $1.600$ \\
$\hE\widehat H_{nSV}^{-1}(q)\approx$    &  $-2.882$  &  $-2.087$  &  $-1.622$  &  $1.045$  &  $1.277$  &  $1.577$ \\
$\hS\widehat H_{nSV}^{-1}(q)\approx$    &  $0.083$  &  $0.095$  &  $0.092$  &  $0.050$  &  $0.073$  &  $0.127$ \\
$\hE\widetilde F_{nSV}^{-1}(q)\approx$   &  $-4.782$  &  $-2.769$  &  $-1.944$  &  $1.079$  &  $1.347$  &  $1.809$ \\
$\hS\widetilde F_{nSV}^{-1}(q)\approx$  &  $1.184$  &  $0.651$  &  $0.416$  &  $\phantom{-}0.041$  &  $\phantom{-}0.058$  &  $\phantom{-}0.097$ \\
\boldhline
\end{tabular} 
}
\end{table}

Let us take the populations of Section \ref{s:simul_out} with $p/N=0.06$. In this case of outliers, the corresponding to Tables \ref{t3}--\ref{t6} results are given in Tables \ref{t7}--\ref{t10}. A behaviour of the approximations to the distributions is very similar to that in the case of no outliers, but erorrs of the approximations are larger now. One can observe also that the estimates of the true Edgeworth expansions, which use the auxiliary information, are much more biased. It holds for the alternative empirical Edgeworth approximations too but in the case of the statistic $U_V$ only (Tables \ref{t8} and \ref{t10}). A sensitivity to the outliers is the smallest comparing Table \ref{t9} with Table \ref{t5}.

\section{Summary}\label{s:7}

The specific estimation strategies for scales are considered under simple random samples without replacement. In a sense, they are consistent with the scale estimation by the sample variance. In particular, the proposed strategies ($S_1$) and ($S_2$) combine the use of the GMD statistic and its bias correction. This combination allows an improvement of the scale estimation in populations where the part of outliers is not large. As the numerical modeling indicates too, under ideal for the sample variance conditions (when there are no outliers), the efficiency of the strategies is not worse. It is important robustness property. 

The new estimators of the parameters and also empirical Edgeworth expansions for the GMD statistic and the sample variance are proposed using the detailed decompositions of the statistics. In general, well-correlated auxiliary information leads to effective inferences about the statistics of interest.

%\small
%\bibliographystyle{abbrv}
%\bibliography{liter_abc}

\end{document}